\documentclass{amsart}

\newtheorem{theorem}{Theorem}[section]
\newtheorem{lemma}[theorem]{Lemma}
\newtheorem{proposition}[theorem]{Proposition}
\newtheorem{corollary}[theorem]{Corollary}

\theoremstyle{definition}

\newtheorem{remark}[theorem]{Remark}

\newcommand{\Z}{\mathbb{Z}}

\newcommand{\ep}{\text{exp}}

\usepackage{amscd,amssymb}
\begin{document}

\title[The Nowicki Conjecture
for relatively free algebras]
{The Nowicki Conjecture\\
for relatively free algebras}
\author[Lucio Centrone, {\c S}ehmus F{\i}nd{\i}k]
{Lucio Centrone and {\c S}ehmus F{\i}nd{\i}k}
\address{IMECC, Universidade Estadual de Campinas,
Rua Sergio Buarque de Holanda 651,13083-859,
Campinas (SP), Brazil}
\email{centrone@ime.unicamp.br}
\address{Department of Mathematics,
\c{C}ukurova University, 01330 Balcal\i,
 Adana, Turkey}
\email{sfindik@cu.edu.tr}

\thanks
{The first author is partially supported by FAPESP grant 2018/02108-7 and by CNPq.}
\thanks
{The research of the second named author was partially supported by \c{C}ukurova University (BAP/10746)
for Bilateral Scientific Cooperation between Turkey and Brazil.}


\begin{abstract}
A linear locally nilpotent derivation of the polynomial algebra $K[X_m]$ in $m$ variables over a field $K$ of characteristic 0
is called a Weitzenb\"ock derivation. It is well known from the classical theorem of Weitzenb\"ock
that the algebra of constants $K[X_{m}]^{\delta}$ of a Weitzenb\"ock derivation $\delta$ is finitely generated.
Assume that $\delta$ acts on the polynomial algebra $K[X_{2d}]$ in $2d$ variables as follows:
$\delta(x_{2i})=x_{2i-1}$, $\delta(x_{2i-1})=0$, $i=1,\ldots,d$.
The Nowicki conjecture states that the algebra $K[X_{2d}]^{\delta}$
is generated by $x_1,x_3.\ldots,x_{2d-1}$, and $x_{2i-1}x_{2j}-x_{2i}x_{2j-1}$, $1\leq i<j\leq d$.
The conjecture was proved by several authors based on different techniques.
We apply the same idea to two relatively free algebras of rank $2d$.
We give the infinite set of generators of the algebra of constants in the
the free metabelian associative algebras $F_{2d}(\mathfrak A)$,
and finite set of generators in 
the free algebra $F_{2d}(\mathcal G)$ in the variety determined by the identities of
the infinite dimensional Grassmann algebra.

\end{abstract}

\maketitle

\section{Introduction}

Let $K$ be a field of characteristic zero, $X_m=\{x_1,\ldots,x_m\}$ be a set of variables,
and $KX_m$ be the vector space with basis $X_m$. Now consider a non-zero 
nilpotent linear operator $\delta$ of $KX_m$.
Certainly $\delta$  can be extended to a derivation
of the polynomial algebra $K[X_m]$, the free algebra of rank $m$ in the class of
all commutative unitary algebras. Such derivations are called Weitzenb\"ock due to the classical
theorem of Weitzenb\"ock \cite{W}, dating back to 1932, which states that the
algebra $K[X_m]^{\delta}=\text{ker}\delta$ of constants of the derivation $\delta$ in the algebra
$K[X_m]$ is
finitely generated.
Obviously, any Weitzenb\"ock derivation ${\delta}$ is locally nilpotent and the linear operator
$\text{exp}(\delta)$ which acts on the vector space $KX_m$
is unipotent. Hence the algebra $K[X_m]^{\delta}$ of constants of $\delta$ 
is equal to the algebra of invariants
\[
K[X_m]^{\text{exp}(\delta)}=\{x\in K[X_m]\mid \text{exp}(\delta)(x)=x\}.
\]
The algebra of invariants of $\text{exp}(\delta)$ is equal to the algebra of invariants $K[X_m]^{G}$, where the group $G$ consists of all elements $\text{exp}(c\delta)$, $c\in K$ and is isomorphic to the unitriangular group.
This means that the methods inherited from the classical invariant theory
are linked to the study of the algebra of constants $K[X_m]^{\delta}$.
One can see computational aspects of algebra of constants and invariant theory in the books by Nowicki \cite{N}, Derksen and Kemper \cite{DK},
and Sturmfels \cite{St}.

The Jordan normal form of the linear operator $\delta$ consists of Jordan cells with zero diagonals.
The Weitzenb\"ock derivations are in one-to-one correspondence
with the partition of $m$, up to a linear change of variables.
Thus there are essentially finite number of  Weitzenb\"ock derivations for a fixed dimension $m$.
In particular, let $m=2d$, $d\geq1$, and assume that the Jordan form of $\delta$ contains
the Jordan cells of size $2\times2$ only, i.e.,
\[
J(\delta)=\left(\begin{matrix}
0&1&\cdots&0&0\\
0&0&\cdots&0&0\\
\vdots&\vdots&\ddots&\vdots&\vdots\\
0&0&\cdots&0&1\\
0&0&\cdots&0&0
\end{matrix}\right).
\]
We assume that $\delta(x_{2j})=x_{2j-1}$, $\delta(x_{2j-1})=0$, $j=1,\ldots,d$.
Nowicki conjectured in his book \cite{N} (Section 6.9, page 76) that in this case the algebra $K[X_{2d}]^{\delta}$
is generated by $x_1,x_3.\ldots,x_{2d-1}$, and $x_{2i-1}x_{2j}-x_{2i}x_{2j-1}$, $1\leq i<j\leq d$.

The conjecture was proved by several authors
with different techniques.
The first published proof appeared in 2004 by Khoury \cite{K1} in his PhD thesis,
followed by his paper \cite{K2}, where he makes use of a computational approach involving Gr\"obner basis techniques. Derksen and Panyushev applied ideas of classical invariant
theory in order to prove the Nowicki conjecture but their proof remained unpublished. Later in 2009, 
Drensky and Makar-Limanov \cite{DML} confirmed the conjecture
by an elementary proof from undergraduate algebra, without involving invariant theory.
Another simple short proof was given by Kuroda \cite{Kuroda} in the same year.
Bedratyuk \cite{Bed} proved the Nowicki conjecture by reducing it to a well known problem of classical invariant theory.

The following problem arises naturally.
Let $\delta$ be a Weitzenb\"ock derivation of the free algebra
$F_{m}(\mathfrak V)$ of rank $m$ in a given variety
$\mathfrak V$ (or relatively free algebra of rank $m$ of $\mathfrak{V}$), then give an explicit set of generators of
the algebra of constants $F_{m}^{\delta}(\mathfrak V)$.
Recently, Dangovski et al. \cite{DDF,DDF1} gave some results in this direction. They showed the algebra of constants $F_{m}^{\delta}(\mathfrak V)$ is not finitely generated as an algebra
except for some trivial cases, when $\mathfrak {V}$ is the variety of metabelian Lie algebras or the variety of metabelian associative algebras.
Although this, the algebra of constants $(F_d'(\mathfrak{V}))^\delta$ in the commutator ideal $F_d'(\mathfrak{V})$ of $F_d(\mathfrak{V})$
is a finitely generated module of $K[X_d]^{\delta}$ and $K[U_d,V_d]^{\delta}$, respectively.
Here $\delta$ acts on $U_d$ and $V_d$ in the same way as on $X_d$.  

In \cite{DG} Drensky and Gupta studied Weitzenb\"ock derivations $\delta$ acting on $F_m(\mathfrak{V})$ proving that if the algebra $UT_2(K)$ of $2\times 2$ upper triangular matrices over a field $K$ of characteristic zero belongs to a variety $\mathfrak{V}$, then $F_m^\delta(\mathfrak{V})$ is not finitely generated whereas if $UT_2(K)$ does not belong to $\mathfrak{V}$, then $F_m^\delta(\mathfrak{V})$ is finitely generated by a result of Drensky in \cite{D1}.

When studying varieties of associative unitary algebras over a field of characteristic zero, the polynomial identities of $UT_2(K)$ play a crucial role. It is well known the identities of $UT_2(K)$ follow from the metabelian identity $[x_1,x_2][x_3,x_4]\equiv0$, so every variety $\mathfrak{V}$ contains $UT_2(K)$ or satisfies the Engel identity $[x_2,x_1,\ldots,x_1]\equiv0$. Another precious tool in the study of varieties is given by the identities of the infinitely generated Grassmann algebra $G$. In a famous work by Kemer (see \cite{kem1} or his monograph \cite{kem2}) the author proves that any variety $\mathfrak{V}$ of associative algebras satisfies the Specht property, i.e, any proper subvariety of $\mathfrak{V}$ is finitely generated. In particular the identities of $\mathfrak{V}$ are the same of the identities of the so called Grassmann envelope of a finite dimensional superalgebra. 

In the description of varieties of (not necessarily associative) algebras and their rate of growth the exponent of an algebra is worthy used. In the next lines we shall recall the definition of the exponent. Let $\mathfrak{V}$ be a variety, then by a well known multilinearization process, in order to study the polynomial identities of $\mathfrak{V}$, it is enough studying its multilinear identities. Hence, if we denote by $P_n$ the vector space of multilinear polynomials of degree $n$, we denote by $c_n(\mathfrak{V})$ the dimension of $P_n$ modulo the multilinear identities of $\mathfrak{V}$ of degree $n$. In \cite{giz1}, \cite{giz2} Giambruno and Zaicev proved that if $\mathfrak{V}$ is a variety of associative algebras, then there exists the limit
\[\text{\rm exp}(\mathfrak{V})=\lim_{n\rightarrow\infty}\sqrt[n]{c_n(\mathfrak{V})}\] and it is a non-negative integer called the \textit{PI-exponent} of $\mathfrak{V}$ denoted by $\ep(\mathfrak{V})$. In the Lie case the existence of the exponent has been proved in the finite dimensional case by Zaicev in \cite{Z1}. Indeed varieties of exponent one are called polynomial growth varieties, hence varieties having a non-polynomial growth may have exponent at least two. The algebras $UT_2(K)$ and $G$ are also fundamental tools in the study of the growth of varieties of exponent two. In particular in \cite{giz3} the authors constructed five ``minimal'' varieties $\{A_i\}_{i=1,\ldots,5}$ and they proved that a variety $\mathfrak{V}$ has exponent two if and only if none of the $A_i$ belongs to $\mathfrak{V}$ and either $G$ or $UT_2(K)$ belongs to $\mathfrak{V}$.

At the light of the above discussion, the goal of the paper is giving experimental results supporting the computational aspects of Nowicki conjecture for any $2d$-generated relatively free algebra of any variety of associative algebras.
In the second section we give the full set of infinite generators of the free metabelian associative algebra of rank $2d$, $d\geq1$,
whose commutator ideal has a $K[X_{2d}]$-bimodule structure. In this case its algebra of constants is a
$K[U_{2d},V_{2d}]^{\delta}$-module, where the elements $u_j$ and $v_j$ stand for the associative (left side) and Lie (right side)
multiplications, respectively. This allows to use the results from Nowicki conjecture as a counterpart of the work.
For this purpose, we worked in the abelian wreath product where the algebra of constants is embedded into,
and $\delta$ acts in the same way, and finally the results were pulled back to the algebra of constants.

In Section 3, we give the explicit form of the set of generators for the algebra of constants in the variety $\mathcal{G}$ of associative algebras generated by the infinite dimensional Grassmann algebra. Of course the free metabelian associative algebra coincides with the variety $\mathcal{V}$ generated by $UT_2(K)$. As noticed above, both $\mathcal{G}$ and $\mathcal{V}$ have exponent two although $G$ does not belong to the variety generated by $UT_2(K)$. At the light of the results \cite{D1} and \cite{DG} the algebra of constants in $\mathcal{G}$ is finitely generated not as in the metabelian case. 

\section{The free metabelian associative algebras}

We start off our investigation with the relatively free algebra of the associative metabelian algebra.
Let $K$ be a fileld of characteristic zero, $P_{2d}$ be the free unitary associative algebra of rank ${2d}$
over $K$, $P_{2d}'=P_{2d}[P_{2d},P_{2d}]P_{2d}$ be its commutator ideal generated by all elements of the form
\[
[x,y]=xy-yx, \quad x,y\in P_{2d}.
\]
Let us consider the quotient
algebra $F_{2d}=P_{2d}/(P_{2d}')^2$. The algebra $F_{2d}$ is the free algebra of rank $2d$ in the
variety of all associative algebras satisfying the polynomial identity $[x,y][z,t]\equiv0$.
Let the free associative metabelian algebra $F_{2d}$ be generated by $X_{2d}=\{x_1,\ldots,x_{2d}\}$.
We assume that all Lie commutators are left normed; i.e., 
\[
[x,y,z]=[[x,y],z], \quad x,y,z\in F_{2d}.
\]
The commutator ideal $F_{2d}'$ of $F_{2d}$ 
has the following basis (see e.g. \cite{B, D})
\[
x_1^{\xi_1}\cdots x_{2d}^{\xi_{2d}}[x_{i},x_{j},x_{j_1},\ldots,x_{j_m}],\quad \xi_i\geq 0, 1\leq i>j\leq j_1\leq\cdots\leq j_m\leq 2d.
\]
As a consequence of the metabelian identity in $F_{2d}'$ we have the following identity:
\[
x_{i_{\pi(1)}}\cdots x_{i_{\pi(m)}}[x_i,x_j,x_{j_{\sigma(1)}},\ldots,x_{j_{\sigma(n)}}]
\equiv x_{i_1}\cdots x_{i_m}[x_{i},x_{j},x_{j_1},\ldots,x_{j_n}],
\]
for any permutation $\pi\in S_m$ and $\sigma\in S_n$. Thus the commutator ideal 
$F_{2d}'$ can be ``seen'' as a polynomial algebra from both sides via the associative (left side) and Lie (right side) multiplication.

We recall now some of the results and constructions given in \cite{DDF1}. Let
$U_{2d}=\{u_1,\ldots,u_{2d}\}$ and $V_{2d}=\{v_1,\ldots,v_{2d}\}$ be two sets of commuting variables
and let $K[U_{2d},V_{2d}]$ be the polynomial algebra acting on $F_{2d}'$ as follows.
If $f\in F_{2d}'$, then
\[
fu_i=x_if, \quad fv_i=[f,x_i], \quad i=1,\ldots,2d.
\]
This action defines a $K[U_d,V_d]$-module structure on the vector space $F_{2d}'$.

We are going to construct a wreath product which is the same as the one used in \cite{DDF1}. It is a particular case
of the construction of Lewin \cite{L} given in \cite{DG} and is similar to the construction of
Shmel'kin \cite{Sh} in the case of free metabelian Lie algebras as appeared in \cite{DDF}.

Let $Y_{2d} = \{y_1, \ldots, y_{2d}\}$ and $V_{2d}'=\{v_1',\ldots,v_{2d}'\}$ be sets of commuting variables and let
$A_{2d}=\{a_1,\ldots,a_{2d}\}$ in $2d$ variables.
Now let $M_{2d}$ be the free $K[U_{2d},V_{2d}']$-module generated by $A_{2d}$
equipped with with trivial multiplication
$M_{2d}\cdot M_{2d}=0$. We endow $M_{2d}$ with a structure of a free $K[Y_{2d}]$-bimodule
structure via the action
\[
y_ja_i = a_iu_j, \quad a_iy_j = a_iv_j', \quad i,j = 1, \ldots, 2d.
\]
The wreath product $W_{2d}=K[Y_{2d}]\rightthreetimes M_{2d}$ 
is an algebra satisfying the metabelian identity. As well as in \cite{L} $F_{2d}$ can be embedded into $W_{2d}$. In fact we have the following result.

\begin{proposition}\label{embedding}
The mapping $\varepsilon :x_j\to y_j+a_j$, $j=1,\ldots,2d$, extends to an embedding $\varepsilon$ of $F_{2d}$ into $W_{2d}$.
\end{proposition}

We identify $v_i=v_i'-u_i$, $i=1,\ldots,2d$, and get
\begin{equation}\label{commutator element}
\varepsilon(x_{i_1}\cdots x_{i_m}[x_{i},x_{j},x_{j_1},\ldots,x_{j_n}])
=(a_iv_j-a_jv_i)v_{j_1}\cdots v_{j_n}u_{i_1}\cdots u_{i_m}
\end{equation}
Thus we may assume that $M_{2d}$ is a free $K[U_{2d},V_{2d}]$-module. Clearly the commutator ideal
$F_{2d}'$ is embedded into $M_{2d}$, too.
An element $\sum a_if(U_{2d},V_{2d})\in M_{2d}$
is an image of some element from $F_{2d}'$ if and only if $\sum v_if(U_{2d},V_{2d})=0$,
as a consequence of (\ref{commutator element}).

Let $\delta$ be the Weitzenb\"ock derivation of $F_{2d}$ acting on the
variables $U_{2d},V_{2d}$, as well as explained in the Introduction on $X_{2d}$. By \cite{DDF1} we know that
the vector space $M_{2d}^{\delta}$ of the constants of $\delta$ in the $K[U_{2d},V_{2d}]$-module $M_{2d}$
is a $K[U_{2d},V_{2d}]^{\delta}$-module.
The following results are particular cases of \cite{D1} (Proposition 3) and \cite{DDF1}.

\begin{theorem}\label{theorem for finite generation}
The vector spaces $(F_{2d}')^{\delta}$ and $M_{2d}^{\delta}$ are finitely generated $K[U_{2d},V_{2d}]^{\delta}$-modules.
\end{theorem}

In the sequel we shall write down an explicit set of generators for the algebra of constants $K[U_{2d},V_{2d}]^{\delta}$ as a consequence of
the Nowicki conjecture \cite{N} (proved in \cite{K1, K2, DML, Bed, Kuroda}), and one of the results by Drensky and Makar-Limanov \cite{DML}. 
Let the Weitzenb\"ock derivation $\delta$ act on $X_{2d},A_{2d},U_{2d},V_{2d}$
by the rule
\[
\delta(x_{2i})=x_{2i-1}, \delta(x_{2i-1})=0,\quad \delta(a_{2i})=a_{2i-1}, \delta(a_{2i-1})=0,
\]
\[
\delta(u_{2i})=u_{2i-1}, \delta(u_{2i-1})=0,\quad\delta(v_{2i})=v_{2i-1}, \delta(v_{2i-1})=0,
\]
for $i=1,\ldots,d$. Then the algebra of constants $K[U_{2d},V_{2d}]^{\delta}$ is generated by 
\[
u_1,u_3,\ldots,u_{2d-1},v_1,v_3,\ldots,v_{2d-1}
\]
and the determinants
\[
\alpha_{pq}=u_{2p-1}u_{2q}-u_{2p}u_{2q-1}=\begin{vmatrix}
u_{2p-1} & u_{2p} \\
u_{2q-1} &u_{2q}
\end{vmatrix},\quad 1\leq p<q\leq d,
\]
\[
\beta_{pq}=v_{2p-1}v_{2q}-v_{2p}v_{2q-1}=\begin{vmatrix}
v_{2p-1} & v_{2p} \\
v_{2q-1} &v_{2q}
\end{vmatrix},\quad1\leq p<q\leq d,
\]
\[
\gamma_{pq}=u_{2p-1}v_{2q}-u_{2p}v_{2q-1}=\begin{vmatrix}
u_{2p-1} & u_{2p} \\
v_{2q-1} &v_{2q}
\end{vmatrix}, \quad p,q=1,\ldots,d,
\]
with the following defining relations
\begin{equation} \label{S1}
u_{2i-1}\alpha_{jk}-u_{2j-1}\alpha_{ik}+u_{2k-1}\alpha_{ij}=0,\quad 1\leq i<j<k\leq d,
\end{equation}
\begin{equation} \label{S2}
u_{2i-1}\gamma_{jk}-u_{2j-1}\gamma_{ik}+v_{2k-1}\alpha_{ij}=0,\, 1\leq i<j\leq d,\, 1\leq k\leq d,
\end{equation}
\begin{equation} \label{S3}
u_{2i-1}\beta_{jk}-v_{2j-1}\gamma_{ik}+v_{2k-1}\gamma_{ij}=0,\, 1\leq i\leq d,\, 1\leq j<k\leq d,
\end{equation}
\begin{equation} \label{S4}
v_{2i-1}\beta_{jk}-v_{2j-1}\beta_{ik}+v_{2k-1}\beta_{ij}=0,\quad 1\leq i<j<k\leq d,
\end{equation}
\begin{equation}  \label{R1}
\alpha_{ij}\alpha_{kl}-\alpha_{ik}\alpha_{jl}+\alpha_{il}\alpha_{jk}=0,\, 1\leq i<j<k<l\leq d,
\end{equation}
\begin{equation}  \label{R2}
\alpha_{ij}\gamma_{kl}-\alpha_{ik}\gamma_{jl}+\gamma_{il}\alpha_{jk}=0,\, 1\leq i<j<k\leq d,\, 1\leq l\leq d,
\end{equation}
\begin{equation}  \label{R3}
\alpha_{ij}\beta_{kl}-\gamma_{ik}\gamma_{jl}+\gamma_{il}\gamma_{jk}=0,\,1\leq i<j\leq d,\,1\leq k<l\leq d,
\end{equation}
\begin{equation}  \label{R4}
\gamma_{ij}\beta_{kl}-\gamma_{ik}\beta_{jl}+\gamma_{il}\beta_{jk}=0,\,1\leq i\leq d,\,1\leq j<k<l\leq d,
\end{equation}
\begin{equation}  \label{R5}
\beta_{ij}\beta_{kl}-\beta_{ik}\beta_{jl}+\beta_{il}\beta_{jk}=0,\quad 1\leq i<j<k<l\leq d.
\end{equation}
The vector space $K[U_{2d},V_{2d}]^{\delta}$
has a {\it canonical} linear basis consisting of the elements of the form
\begin{equation}  \label{B1}
v_{2i_1-1}\cdots v_{2i_m-1}\beta_{p_1q_1}\cdots \beta_{p_rq_r}
\gamma_{p'_1q'_1}\cdots \gamma_{p'_sq'_s}\alpha_{p''_1q''_1}\cdots \alpha_{p''_tq''_t}u_{2j_1-1}\cdots u_{2j_n-1}
\end{equation}
such that among the generators $\beta_{pq}$, $\gamma_{p'q'}$, and $\alpha_{p''q''}$ there is no \textit{intersection},
and no one \textit{covers} $v_{2i_k-1}$ or $u_{2j_l-1}$.

Note that each $\beta_{pq}$, $\gamma_{p'q'}$, and $\alpha_{p''q''}$ is identified with the open interval
$(p+d,q+d)$, $(p',q'+d)$, and $(p'',q'')$, respectively, on the real line. The generators \textit{intersect} each other if the corresponding
open intervals have a nonempty intersection and are not contained in each other. On the other hand the generators
$v_{2i-1}$ and $u_{2j-1}$ are identified with the points $i+d$ and $j$, respectively.
We say also that a generator among $\beta_{pq}$, $\gamma_{p'q'}$, or $\alpha_{p''q''}$ \textit{covers}
$v_{2i-1}$ or $u_{2j-1}$ if the corresponding open interval covers the corresponding point.

The pairs of indices are ordered in the following way: $p_1 \leq \cdots \leq p_r$ and if $p_{\xi}=p_{{\xi}+1}$, then $q_{\xi}\leq q_{{\xi}+1}$;
$p'_1 \leq \cdots \leq p'_s$ and if $p'_{\mu}=p'_{{\mu}+1}$, then $q'_{\mu}\leq q'_{{\mu}+1}$;
$p''_1 \leq \cdots \leq p''_t$ and if $p''_{\sigma}=p''_{{\sigma}+1}$, then $q''_{\sigma}\leq q''_{{\sigma}+1}$;
and $i_1 \leq \cdots \leq i_m$, $j_1 \leq \cdots \leq j_n$.

In order to detect the constants in the commutator ideal $F_{2d}'$ it is sufficient to work in
the $K[U_{2d},V_{2d}]^{\delta}$-submodule $C_{2d}^{\delta}$ of $M_{2d}^{\delta}$, which is generated by 
\[
a_1,a_3,\ldots,a_{2d-1}
\]
and the determinants
\[
w_{pq}=a_{2p-1}v_{2q}-a_{2q}v_{2p-1}=\begin{vmatrix}
a_{2p-1} & a_{2q} \\
v_{2p-1} &v_{2q}
\end{vmatrix}, \quad p,q=1,\ldots,d.
\]
and spanned as a vector space on the elements of the form
\begin{equation}  \label{sp1}
a_{2i_0-1}v_{2i_1-1}\cdots v_{2i_m-1}\beta_{p_1q_1}\cdots \beta_{p_rq_r}
\gamma_{p'_1q'_1}\cdots \gamma_{p'_sq'_s}\alpha_{p''_1q''_1}\cdots \alpha_{p''_tq''_t}u_{2j_1-1}\cdots u_{2j_n-1}
\end{equation}
\begin{equation}  \label{sp2}
w_{p_0q_0}v_{2i_1-1}\cdots v_{2i_m-1}\beta_{p_1q_1}\cdots \beta_{p_rq_r}
\gamma_{p'_1q'_1}\cdots \gamma_{p'_sq'_s}\alpha_{p''_1q''_1}\cdots \alpha_{p''_tq''_t}u_{2j_1-1}\cdots u_{2j_n-1}
\end{equation}
for each $i_0,p_0,q_0=1\ldots,d$.
We also have the following relations in the algebra $C_{2d}^{\delta}$ as a consequence of \cite{DML}.
\begin{equation} \label{S}
a_{2i-1}\beta_{jk}-w_{ik}v_{2j-1}+w_{ij}v_{2j-1}=0,\quad 1\leq i\leq d, \,1\leq j<k\leq d,
\end{equation}
\begin{equation}  \label{R}
w_{ij}\beta_{kl}-w_{ik}\beta_{jl}+w_{il}\beta_{jk}=0,\quad 1\leq i\leq d, \,1\leq j<k<l\leq d.
\end{equation}

We denote by $L$ the $K[U_{2d},V_{2d}]^{\delta}$-submodule of $C_{2d}^{\delta}$
generated by the following elements
\begin{equation} \label{g1}
w_{ii},\quad \quad 1\leq i\leq d,
\end{equation}
\begin{equation} \label{g2}
w_{ij}+w_{ji},\quad 1\leq i<j\leq d,
\end{equation}
\begin{equation}\label{g3}
a_{2i-1}v_{2j-1}-a_{2j-1}v_{2i-1},\quad 1\leq i<j\leq d,
\end{equation}
\begin{equation}\label{g4}
a_{2i-1}\beta_{pq}-w_{pq}v_{2i-1},\, 1\leq i\leq d, \, 1\leq p<q\leq d,
\end{equation}
\begin{equation}\label{g5}
w_{ij}\beta_{pq}-w_{pq}\beta_{ij}, \, 1\leq i<j\leq d,  \, 1\leq p<q\leq d,
\end{equation}
\begin{equation} \label{g6}
a_{2i-1}\beta_{jk}-a_{2j-1}\beta_{ik}+a_{2k-1}\beta_{ij},\quad 1\leq i<j<k\leq d,
\end{equation}
\begin{equation}  \label{g7}
w_{kl}\gamma_{ij}-w_{jl}\gamma_{ik}+w_{jk}\gamma_{il},\,1\leq i\leq d,\,1\leq j<k<l\leq d.
\end{equation}
\begin{equation} \label{g8}
w_{jk}u_{2i-1}-a_{2j-1}\gamma_{ik}+a_{2k-1}\gamma_{ij},\, 1\leq i\leq d,\, 1\leq j<k\leq d.
\end{equation}
One can observe that the generating elements (\ref{g1})-(\ref{g8}) of $L$ are
the images of some elements in the commutator ideal $F_{2d}'$ of the free associative algebra $F_{2d}$.

\begin{lemma}\label{quotient span}
The quotient space $C_{2d}^{\delta}/L$ is spanned on the following elements.
\begin{equation}  \label{sp3}
a_{2i_0-1}v_{2i_1-1}\cdots v_{2i_m-1}\beta_{p_1q_1}\cdots \beta_{p_rq_r}
\gamma_{p'_1q'_1}\cdots \gamma_{p'_sq'_s}\alpha_{p''_1q''_1}\cdots \alpha_{p''_tq''_t}u_{2j_1-1}\cdots u_{2j_n-1}
\end{equation}
\begin{equation}  \label{sp4}
w_{p_0q_0}\beta_{p_1q_1}\cdots \beta_{p_rq_r}
\gamma_{p'_1q'_1}\cdots \gamma_{p'_sq'_s}\alpha_{p''_1q''_1}\cdots \alpha_{p''_tq''_t}u_{2j_1-1}\cdots u_{2j_n-1}
\end{equation}
such that when we replace 
$a_{2i_0-1}$ and $w_{p_0q_0}$ by $v_{2i_0-1}$ and $\beta_{p_0q_0}$, respectively, we obtain the basis elements of the algebra $K[U_{2d},V_{2d}]^{\delta}$.
\end{lemma}

\begin{proof}
We start by showing that the spanning elements of the form (\ref{sp2}) reduce to the form (\ref{sp4}). 
By (\ref{g1}) and (\ref{g2}) we assume $p_0<q_0$. Then (\ref{g7}) ensures that
$\beta_{p_0q_0}$ does not intersect with each $\gamma_{p'_{\xi}q'_{\xi}}$. Now by the relation (\ref{R})
we assume $\beta_{p_0q_0}$ does not intersect with each $\beta_{p_{\mu}q_{\mu}}$.
By definition, $\beta_{p_0q_0}$ does not intersect with each $\alpha_{p_{\sigma}q_{\sigma}}$ and it does not cover
each $u_{2i_c-1}$. By (\ref{g5}) we can fix the order among $\beta_{p_0q_0},\beta_{p_1q_1},\ldots, \beta_{p_rq_r}$.

Now if $m\geq 1$ in (\ref{sp2}), by (\ref{g4}) we may replace 
$w_{p_0q_0}v_{2i_m-1}$ by $a_{2i_m-1}\beta_{p_0q_0}$. Hence we may assume $m=0$, and the expression (\ref{sp2}) reduces to (\ref{sp4})
with the desired conditions.

Let us consider the spanning element of the form  (\ref{sp1}). We have to check whether we obtain a basis element of 
$K[U_{2d},V_{2d}]^{\delta}$ after replacing $a_{2i_0-1}$ by $v_{2i_0-1}$. If $v_{2i_0-1}$
is covered by some $\gamma_{p'_{\xi}q'_{\xi}}$, then by (\ref{g8}) we replace 
$a_{2i_0-1}\gamma_{p'_{\xi}q'_{\xi}}$ by $w_{iq'_{\xi}}u_{2p'_{\xi}-1}+a_{2q'_{\xi}-1}\gamma_{p'_{\xi}i}$.
The elements $v_{2q'_{\xi}-1}$ is not covered by $\gamma_{p'_{\xi}i}$. Additionally we have
\[
w_{iq'_{\xi}}u_{2p'_{\xi}-1}v_{2i_1-1}\cdots v_{2i_m-1}\beta_{p_1q_1}\cdots \beta_{p_rq_r}
\gamma_{p'_1q'_1}\cdots \gamma_{p'_sq'_s}\alpha_{p''_1q''_1}\cdots \alpha_{p''_tq''_t}u_{2j_1-1}\cdots u_{2j_n-1}
\]
which is not of the form (\ref{sp4}). In this expression, if $m=0$, then applying the argument above we reduce it to (\ref{sp4}).
If $m\geq 1$, we apply (\ref{g4}) and get
\[
a_{2i_m-1}v_{2i_1-1}\cdots v_{2i_{m-1}-1}\beta_{iq'_{\xi}}\beta_{p_1q_1}\cdots \beta_{p_rq_r}
\gamma_{p'_1q'_1}\cdots \gamma_{p'_sq'_s}\times
\]
\[
\times\alpha_{p''_1q''_1}\cdots \alpha_{p''_tq''_t}u_{2j_1-1}\cdots u_{2j_n-1}u_{2p'_{\xi}-1}
\]
which has a number of $v_{2i_c-1}$'s and $\gamma_{p'q'}$'s strictly less than in the previous expression. Hence this process terminates when there is no more
$\gamma_{p'_{\xi}q'_{\xi}}$ covering $v_{2i_0-1}$.

Now by (\ref{g6}), we may assume $v_{2i_0-1}$ is not covered by each $\beta_{p_{\xi}q_{\xi}}$.
Then by (\ref{g3}) we are able to fix the order among $v_{2i_0-1},v_{2i_1-1},\ldots, v_{2i_m-1}$.
Finally the order among the generators of constants of $K[U_{2d},V_{2d}]^{\delta}$ is fixed by the defining relations
(\ref{S1})-(\ref{R5}).
\end{proof}

Now we are in the position to prove the main result of this section.

\begin{theorem}
The $K[U_{2d},V_{2d}]^{\delta}$-submodule $L$ of $C_{2d}^{\delta}$ consists of all 
commutator elements; i.e., images of elements in the commutator ideal $F_{2d}'$.
\end{theorem}

\begin{proof}
Let $B_{pq}=\beta_{p_1q_1}\cdots \beta_{p_rq_r}$,
$G_{p'q'}=\gamma_{p'_1q'_1}\cdots \gamma_{p'_sq'_s}$, $A_{p''q''}=\alpha_{p''_1q''_1}\cdots \alpha_{p''_tq''_t}$,
$V_i=v_{2i_1-1}\cdots v_{2i_m-1}$,
$U_j=u_{2j_1-1}\cdots u_{2j_n-1}$,
and let
\[
\sum\pi^{ipqp'q'p''q''j}a_{2i_0-1}V_iB_{pq}G_{p'q'}A_{p''q''}U_j
+
\sum\sigma^{pqp'q'p''q''j}w_{p_0q_0}B_{pq}G_{p'q'}A_{p''q''}U_j
\]
be a commutator element. By the metabelian identity we get
\[
\sum\pi^{ipqp'q'p''q''j}v_{2i_0-1}V_iB_{pq}G_{p'q'}A_{p''q''}U_j
\]\[+
\sum\sigma^{pqp'q'p''q''j}\beta_{p_0q_0}B_{pq}G_{p'q'}A_{p''q''}U_j=0.
\]
On the other hand the two sums are linearly independent basis elements of the algebra of constants
$K[U_{2d},V_{2d}]^{\delta}$. Then by Lemma \ref{quotient span} we have
$\pi^{ipqp'q'p''q''j}=\sigma^{pqp'q'p''q''j}=0$, and  
\[
\sum\pi^{ipqp'q'p''q''j}a_{2i_0-1}V_iB_{pq}G_{p'q'}A_{p''q''}U_j
\]\[+
\sum\sigma^{pqp'q'p''q''j}w_{p_0q_0}B_{pq}G_{p'q'}A_{p''q''}U_j=0,
\]
because the constants $\pi^{ipqp'q'p''q''j}$ and $\sigma^{pqp'q'p''q''j}$ are uniquely determined by the basis elements.
\end{proof}

As a consequence of the previous result we get the full list of generators of $(F'_{2d})^\delta$ as a $K[U_{2d},V_{2d}]$-module.

\begin{corollary}\label{generators in the ideal}
The $K[U_{2d},V_{2d}]^{\delta}$-module $(F_{2d}')^{\delta}$ is generated by the following polynomials:
\[
g_1(i)=[x_{2i-1},x_{2i}],\quad \quad 1\leq i\leq d,
\]
\[
g_2(i,j)=[x_{2i-1},x_{2j-1}],\quad 1\leq i<j\leq d,
\]
\[
g_3(i,j)=[x_{2i-1},x_{2j}]+[x_{2j-1},x_{2i}],\quad 1\leq i<j\leq d,
\]
\[
g_4(i,p,q)=[x_{2i-1},x_{2p-1},x_{2q}]-[x_{2i-1},x_{2p},x_{2q-1}],\, 1\leq i\leq d, \, 1\leq p<q\leq d,
\]
\[
g_5(i,j,k)=[x_{2i-1},x_{2j-1},x_{2k}]-[x_{2i-1},x_{2k-1},x_{2j}]+[x_{2j-1},x_{2k-1},x_{2i}],
\]
\[
1\leq i<j<k\leq d,
\]
\[
g_6(i,j,p,q)=[x_{2i-1},x_{2p-1},x_{2j},x_{2q}]+[x_{2i},x_{2p},x_{2j-1},x_{2q-1}]
\]
\[
-[x_{2i-1},x_{2p},x_{2j},x_{2q-1}]-[x_{2i},x_{2p-1},x_{2j-1},x_{2q}],
\]
\[
1\leq i<j\leq d,\,1\leq p<q\leq d,
\]
\[
g_7(i,j,k,l)=x_{2i}[x_{2j-1},x_{2k-1},x_{2l}]+x_{2i-1}[x_{2j},x_{2k},x_{2l-1}]
\]
\[
-x_{2i}[x_{2j-1},x_{2k},x_{2l-1}]-x_{2i-1}[x_{2j},x_{2k-1},x_{2l}]
\]
\[
1\leq i\leq d,\,1\leq j<k<l\leq d,
\]
\[
g_8(i,j,k)=x_{2i}[x_{2j-1},x_{2k-1}]-x_{2i-1}[x_{2j},x_{2k-1}],
\]
\[
1\leq i\leq d,\, 1\leq j<k\leq d.
\]
\end{corollary}

We have to add for the generating set of the whole algebra
$(F_{2d})^\delta$ the constants $x_{2i-1}$ and $x_{2i-1}x_{2j}-x_{2i}x_{2j-1}$, $1\leq i<j\leq d$, 
which are needed for the generation
of the factor algebra of $(F_{2d})^\delta$ modulo the commutator ideal of $F_{2d}$.
These generators are the ones lifted from the algebra $K[X_{2d}]^{\delta}$ of constants
of the polynomial algebra to the algebra $(F_{2d})^\delta$  by the fact stated in Corollary 4.3 of the paper \cite{DG} by Drensky and Gupta.
The following result gives an infinite generating set of the subalgebra of constants $(F_{2d})^\delta$ of the free metabelian associative algebra $F_{2d}$ as an algebra.

\begin{corollary}
The algebra $(F_{2d})^\delta$ of the constants is generated by
\[
x_1,x_3,\ldots,x_{2d-1},
\]
\[
x_{2i-1}x_{2j}-x_{2i}x_{2j-1},
\]
\[
g_1f_1,\ldots,g_8f_8,
\]
where $1\leq i<j\leq d$, $f_1,\ldots,f_8\in K[U_{2d},V_{2d}]^{\delta}$, and $g_i$s are as in Corollary \ref{generators in the ideal}.
\end{corollary}

\section{The variety $\mathcal{G}$ generated by the Grassmann algebra}
The variety $\mathcal{G}$ consists of all associative unitary algebras satisfying
the polynomial identity $[z_1,z_2,z_3]\equiv0$.
We shall set $F_{2l}:=F_{2l}(\mathcal{G})$ for $0\leq l\leq d$.
As noted in the Introduction, $\mathcal{G}$ is the variety generated by
the infinite dimensional Grassmann algebra $G$ and $F_{2d}$ coincides
with the $2d$-generated relatively free algebra in the variety $\mathcal{G}$. 
The identities of $G$ and related topics have been studied by several authors.
See for example the paper \cite{KR} by Krakovski and Regev about the ideal of polynomial identities of $G$.

Let $X=\{x_1,x_2,\ldots,x_{d}\}$, $Y=\{y_1,y_2,\ldots,y_{d}\}$ two disjoint sets of variables. Of course the relatively free algebra of $\mathfrak{G}$ of rank $2d$ in the variables from $U:=X\cup Y$ is isomorphic to $F_{2d}$ and we consider the following order inside $U$: \[x_1<y_1<x_2<\cdots<x_d<y_d.\] 

Moreover we say a polynomial $f$  is \textit{homogeneous in the set of variables} $S=\{u_{1},\ldots,u_s\}$ if for each monomial $m$ appearing in $f$ we have $\sum_{i=1}^s\deg_{u_i}m$ is the same.

It is well known (see for example Theorem 5.1.2 of \cite{D}) $F_{2d}$ has a basis consisting of all
\[
x_1^{a_1}y_1^{b_1}\cdots x_{d}^{a_{2d}}y_{d}^{b_{2d}}[u_{i_1},u_{i_2}]\cdots[u_{i_{2c-1}},u_{i_{2c}}],\]\[a_i,b_j\geq0,\ ,u_{i_l}\in U,\ u_{i_1}<u_{i_2}<\cdots<u_{i_{2c}},\ c\geq0.\]

We recall the identity $[z_1,z_2,z_3]\equiv0$ implies the identity
\begin{equation}\label{proof Grass}
[z_1,z_2][z_3,z_4]\equiv-[z_1,z_3][z_2,z_4]
\end{equation}
in $F_{2d}$.

Consider the following Weitzenb\"ock derivation $\delta$ of $F_{2d}$ acting on 
 $U$ such that
\[
\delta(y_{i})=x_{i}, \ \delta(x_i)=0,\quad 1\leq i\leq d.
\]

Our goal is to exhibit a finite set of generators for $F_{2d}^\delta$ and we are going to prove it using an induction argument similar to the one used by Drensky and Makar-Limanov in \cite{DML}.

Let $\alpha=(\alpha_1,\cdots,\alpha_{d-1})\in K^{d-1}$ and consider the algebra endomomorphism $\phi_{\alpha}$ of $F_{2d}$ such that \[\phi_\alpha(x_i)=x_i,\ \ \phi_\alpha(y_i)=y_i,\ \ i=1,\ldots,d-1,\]
\[\phi_\alpha(x_d)=\sum_{i=1}^{d-1}\alpha_ix_i,\ \ \phi_\alpha(y_d)=\sum_{i=1}^{d-1}\alpha_iy_i.\] Notice that $\phi_\alpha$ commutes with $\delta$. Hence if $f\in F_{2d}^\delta$, then $\phi_\alpha(f)\in F_{2d-2}^\delta$ too.

We have the next result on the purpose and, as above, we consider $F_{2l}$ as generated by the disjoint sets $\{x_1,\ldots,x_l\}$ and $\{y_1,\ldots,y_l\}$.

\begin{lemma}\label{usefulreduction}
Let $d\geq2$ and $f\in F_{2d}$ being homogeneous with respect to the set $\{x_d,y_d\}$. If $\phi_\alpha(f)=0$ for some non-zero $\alpha\in K^{d-1}$, then $f$ is in the left ideal generated by \[\omega_\alpha:=\left(\sum_{i=1}^{d-1}\alpha_ix_i\right)y_d-\left(\sum_{i=1}^{d-1}\alpha_iy_i\right)x_d,\ \ \ [\omega_\alpha,u],\ \ u\in U,\]\[\mu_\alpha:=\left[x_d,\sum_{i=1}^{d-1}\alpha_ix_i\right],\ \ \nu_\alpha:=\left[y_d,\sum_{i=1}^{d-1}\alpha_iy_i\right]\]
\end{lemma}
\begin{proof}
Before starting the proof remark $\phi(\omega_\alpha)=\phi(\mu_\alpha)=\phi(\nu_\alpha)=0$. We shall prove the assertion by induction on $\deg_{x_d,y_d}f$. Indeed the result is trivial if $\deg_{x_d,y_d}f=0$, so we may assume the result true for $\deg_{x_d,y_d}f>0$. Let $f$ be as above, then we can write $f$ as 
\[f=a_py_d^p+\sum_{u_0}a'_{p,u_0}[u_0,y_d]y_d^{p-1}+a_{p-1}y_d^{p-1}x_d+\sum_{u^1_1}a'_{p-1,u^1_1}[u^1_1,x_d]y_d^{p-1}\]\[+\sum_{u^2_1}a^{''}_{p-1,u^2_1}[u^2_1,y_d]y_d^{p-2}x_d+a^{'''}_{p-1}[x_d,y_d]y_d^{p-2}+\cdots+a_0x_d^p+\sum_{u_p}a'_{0}[u_p,x_d]x_d^{p-1},\] where the $u^i_j$'s are in $U$ whereas the $a_j$'s, the $a'_{r,u^i_j}$'s, the $a^{''}_{s,u^i_j}$ and the $a^{'''}_{t,u^i_j}$'s all belong to $F_{2d-2}$. Let $p\geq3$ and let us denote \[g_1:=a_py_d^{p-1}+\sum_{u_0}a'_{p,u_0}[u_0,y_d]y_d^{p-2}+\sum_{u^1_1}a'_{p-1,u^1_1}[u^1_1,x_d]y_d^{p-2}+a^{'''}_{p-1}[x_d,y_d]y_d^{p-3}\] and \[g'_2:=f-g_1.\] 
Remark $g'_2$ can be written as \[g'_2=g_2x_d\] for some polynomial $g_2\in F_{2d}$. We consider now \[f\left(\sum_{i=1}^{d-1}\alpha_ix_i\right)\]\[=g_1\omega_\alpha+g_1\left(\left(\sum_{i=1}^{d-1}\alpha_iy_i\right) x_d-\sum_{i=1}^{d-1}\alpha_i\left[x_i,y_d\right]\right)+g^*_2x_d+\sum_{i=1}^{d-1}\alpha_ig_2\left[x_d,x_i\right]\]
\[=g_1\omega_\alpha-a_py_d^{p-1}\left(\left(\sum_{i=1}^{d-1}\alpha_iy_i\right)x_d-\sum_{i=1}^{d-1}\left[x_i,y_d\right]\right)+g'_2x_d+\sum_{i=1}^{d-1}\alpha_ig_2\left[x_d,x_i\right],\]
where $g'_2=g_1\sum_{i=1}^{d-1}\alpha_iy_i+g_2^*$ and $g_2^*=g_2\sum_{i=1}^{d-1}\alpha_ix_i$. Because $p-1>0$ we may apply the same argument above to $f\left(\sum_{i=1}^{d-1}\alpha_ix_i\right)$. Then using (\ref{proof Grass}) we have \[f\left(\sum_{i=1}^{d-1}\alpha_ix_i\right)^2=h_1\omega_\alpha+\sum_{i=1}^{d-1}h'_1[\omega_\alpha,x_i]+h_2x_d+h_3\mu_\alpha.\] 
 Therefore we get 
\[\phi_\alpha(h_2x_d)=0,\] i.e., $\phi_\alpha(h_2)=0$ because clearly $\phi_\alpha(x_d)\neq 0$. Now we can apply induction on $h_2$ obtaining $f\left(\sum_{i=1}^{d-1}\alpha_ix_i\right)^2$ is in the left ideal generated by $\omega_\alpha$, $[\omega_\alpha,u]$, $u\in U$, $\mu_\alpha$ and $\nu_\alpha$. Also remark $\omega_\alpha$, $[\omega_\alpha,u]$, $\mu_\alpha$ and $\nu_\alpha$ are irreducible in $F_{2d}$ and $\sum_{i=1}^{d-1}\alpha_ix_i$ does not depend on $x_d,y_d$, then we conclude $f$ is as desired and we are done. The other cases may be treated similarly also multiplying everithing by $\sum_{i=1}^{d-1}\alpha_iy_i$.
\end{proof}

The next is an easy consequence of Lemma \ref{usefulreduction} and its proof follows verbatim the one of Corollary 2 of \cite{DML}.

\begin{corollary}\label{usefulcorollary}
If $\phi_\alpha(f)=0$ for all non-zero $\alpha=(\alpha_1,\ldots,\alpha_{d-1})\in K^{d-1}$, then $f=0$ if $d>2$ and $f$ is in the left ideal generated by $v_{12}$ and $[v_{12},x_2]$ if $d=2$.
\end{corollary}

We define the following objects inside
$F_{2d}$:
\[
v_{ij}:=x_iy_j-y_ix_j, \quad 1\leq i,j\leq d,
\]
\[w_{ijk}:=y_i[x_j,x_k]-x_i[y_j,x_k], \quad 1\leq i,j,k\leq d,\]
\[z_{ijkl}:=y_i[x_j,v_{k,l}]-x_i[y_j,v_{k,l}], \quad 1\leq i,j,k,l\leq d,\]

Notice that the $v_{ij}$'s, the $w_{ijk}$'s and the $z_{ijkl}$'s are constants of $F_{2d}$ and starting from these objects we shall construct three subsets of elements of $F_{2d}$. We set \[V:=\{v_{ij}|1\leq i,j\leq d\},\]\[W_0:=\{w_{ijk}|1\leq i,j\leq k\leq d\}\] and \[Z_0:=\{z_{ijkl}|1\leq i\leq j\leq k\leq l\leq d\}.\] Suppose $s>0$ and we set \[W_s:=\{y[x,w]-x[y,w]|w\in W_{s-1},\ x\in X,\ y\in Y\},\] \[Z_s:=\{y[x,z]-x[y,z]|z\in Z_{s-1}\ x\in X,\ y\in Y\}.\]
As above, both the $W_s$'s and the $Z_s$'s are subsets of constants. Moreover it can be easily seen for $s\geq d$ both $V_s$ and $W_s$ are 0. We shall denote by $\mathcal{C}$ the algebra generated by the non-zero $X,V,W_l,Z_l$.

\begin{remark}
We have \[[v_{12},x_2]=x_1v_{22}+w_{212}.\]
\end{remark}

The next easy-to-prove relation is crucial in the sequel: \[\delta(y^b)=bxy^{b-1}+\frac{b(b-1)}{2}y^{b-2}[y,x].\]

Before writing the proof of the main result we still need another technical lemma.

\begin{lemma}\label{finalyes}
For every $p\geq1$ we have \[x_{k_1}\cdots x_{k_p}y_d^p=c+g_1(x_1,y_1,\ldots,x_d,y_d)x_d+g_2(x_1,y_1,\ldots,x_d,y_d)[x_d,y_d],\] where $c\in\mathcal{C}$ and $g_1(x_1,y_1,\ldots,x_d,y_d)$, $g_1(x_1,y_1,\ldots,x_d,y_d)\in F_{2d}$.
\end{lemma}
\proof
We shall prove the result by induction on $p$. If $p=1$, then $x_ky_d=v_{kd}+y_kx_d$ and we are done because $v_{kd}\in V$. Suppose the result true for $p\geq 1$ and let us prove it for $p+1$. Notice that \[x_{k_1}\cdots x_{k_{p+1}}y_d^{p+1}=x_{k_1}(x_{k_2}\cdots x_{k_{p+1}}y_d^p)y_d\] and we apply induction on the term inside the parenthesis. We have now \[x_{k_1}\cdots x_{k_{p+1}}y_d^{p+1}=x_{k_1}(c+g_1x_d+g_2[x_d,y_d])y_d\]\[=x_{k_1}cy_d+x_{k_1}g_1x_dy_d+x_{k_1}g_2y_d[x_d,y_d]\]\[=cx_{k_1}y_d+[x_{k_1},c]y_d+x_{k_1}g_1y_dx_d+g[x_d,y_d].\] Now observe that \[x_{k_1}y_d=v_{kd}+y_kx_d\] and \[[x_{k_1},c]y_d=[x_{k_1},c]y_d-[y_{k_1},c]x_d-[y_{k_1},c]x_d=\omega-[y_{k_1},c]x_d,\] where $\omega:=[x_{k_1},c]y_d-[y_{k_1},c]x_d$ belongs to $\mathcal{C}$ too and we are done.
\endproof

\begin{remark}\label{remark}
We shall use in the sequel the following relation \[[x_{k_1},c]y_d=\omega-[y_{k_1},c]x_d,\] where $c,\omega\in\mathcal{C}$.
\end{remark}

\begin{theorem}
The algebra of constants $F^\delta_{2d}$ is finitely generated as an algebra, and its generators are:
\begin{equation}\label{grass1} X;\end{equation}\begin{equation}\label{grass2} V,\end{equation}
\begin{equation}\label{grass3} W_{l},\ l\leq d-1; \end{equation}
\begin{equation}\label{grass4} Z_{l},\ l\leq d-1. \end{equation}
\end{theorem} 

\begin{proof}
It is easy to see that the elements from (\ref{grass1})-(\ref{grass4}) are constants of ${F}_{2d}$.

Let $f=f(X',Y',x_d,y_d)\in F_{2d}^\delta$, then its $U$-homogeneous components lie in $F_{2d}^\delta$ too. Hence we may assume $f$ being homogeneous in $U$. Arguing analogously we can take $f$ being homogeneous in $\{x_d,y_d\}$. We shall prove the theorem by induction on $d$ and on the total degree with respect to the set $\{x_d,y_d\}$. Suppose $d=1$, then \[f=\sum_{k=0}^n\alpha_k x^ky^{n-k}+\sum_{j=1}^{n-1}\beta_j x^{j-1}y^{n-j-1}[x,y].\] This means \[0=\delta(f)=\sum_{k=0}^{n-1}\alpha_k(n-k)x^{k+1}y^{n-k-1}\]\[+[x,y]\sum_{k=1}^{n-2}(n-k-1)\left(\alpha_k\frac{n-k}{2}x^{k-1}y^{n-k-1}+\beta_k x^ky^{n-k-2}\right).\] The previous relation gives us $\alpha_k=0$ for $k\leq n-1$, then $\beta_j=0$ for $1\leq j\leq n-2$. Thus the only possibility is $f=\alpha_nx^n+\beta_{n-1}x^{n-2}[x,y]$ and we are done.

Assume now $d>1$ and the result true for $F_{2d-2}$. We recall $F_{2d-2}$ is generated by the sets of variables $X'=\{x_1,x_2,\ldots,x_{d-1}\}$, $Y'=\{y_1,y_2,\ldots,y_{d-1}\}$ and we set $U'=X'\cup Y'$. If $v\in F_{2d-2}^\delta$ is homogeneous in $U$, then $\deg_{X'}(v)\geq\deg_{Y'}(v)$ and if $\deg_{x_d,y_d}(f)=0$, then $f\in F_{2d-2}^\delta$, so we shall study only the case $\deg_{x_d,y_d}(f)>0$.

Remembering that inside $F_{2d-2}$ every commutator belongs to the center of $F_{2d-2}$, we can write $f$ in the following way:\[f=a_py_d^p+\sum_{u_0}a'_{p,u_0}[u_0,y_d]y_d^{p-1}+a_{p-1}x_dy_d^{p-1}+\sum_{u^1_1}a'_{p-1,u^1_1}[u^1_1,x_d]y_d^{p-1}\]\[+\sum_{u^2_1}a^{''}_{p-1,u^2_1}[u^2_1,y_d]x_dy_d^{p-2}+a^{'''}_{p-1}[x_d,y_d]y_d^{p-2}+\cdots+a_0x_d^p+\sum_{u_p}a'_{0}[u_p,x_d]x_d^{p-1},\] where the $u^i_j$'s are in $U$ whereas the $a_j$'s, the $a'_{r,u^i_j}$'s, the $a^{''}_{s,u^i_j}$ and the $a^{'''}_{t,u^i_j}$'s all belong to $F_{2d-2}$. We derive $f$ and we obtain 
\begin{equation}\label{derivataf}\begin{split}
\delta(f)=\delta(a_p)y_d^p+pa_px_dy_d^{p-1}+\frac{p(p-1)}{2}a_p[y_d,x_d]y_d^{p-2}\\
+\sum_{u_0}\left\{\delta(a'_{p,u_0})[u_0,y_d]y_d^{p-1}+a'_{p,u_0}[\delta(u_0),y_d]y_d^{p-1}+a'_{p,u_0}[u_0,x_d]y_d^{p-1}\right.\\\left.+(p-1)a'_{p,u_0}[u_0,x_d]x_dy_d^{p-2}\right\}
+\delta(a_{p-1})x_dy_d^{p-1}+(p-1)a_{p-1}x_d^2y_d^{p-2}\\+\frac{(p-1)(p-2)}{2}a_{p-1}[y_d,x_d]x_dy_d^{p-3}+\sum_{u^1_1}\left\{\delta(a'_{p-1,u^1_1})[u^1_1,x_d]y_d^{p-1}\right.\\\left.+a'_{p-1,u^1_1}[\delta(u^1_1),x_d]y_d^{p-1}+(p-1)a'_{p-1,u^1_1}[u^1_1,x_d]x_dy_d^{p-2}\right\}\\
+\sum_{u^2_1}\left\{\delta(a^{''}_{p-1,u^2_1})[u^2_1,y_d]x_dy_d^{p-2}+a^{''}_{p-1,u^2_1}[\delta(u^2_1),y_d]x_dy_d^{p-2}+a^{''}_{p-1,u^2_1}[u^2_1,x_d]x_dy_d^{p-2}\right.\\\left.+(p-2)a^{''}_{p-1,u^2_1}[u^2_1,y_d]x_d^2y_d^{p-3}\right\}+\delta(a^{'''}_{p-1})[x_d,y_d]y_d^{p-2}+(p-2)a^{'''}_{p-1}[x_d,y_d]x_dy_d^{p-3}\\+\cdots+\delta(a_0)x_d^p+\sum_{u_p}\left\{\delta(a'_0)[u_p,x_d]x_d^{p-1}+a'_0[\delta(u_p),x_d]x_d^{p-1}\right\}=0
\end{split}\end{equation}

By (\ref{derivataf}) we immediately get $\delta(a_p)=0$, hence:

\begin{equation}\label{obs1}
a_p\in F_{2d}^\delta. 
\end{equation}
We also have $\frac{p(p-1)}{2}a_p+\delta(a^{'''}_{p-1})=0$, then by (\ref{obs1}) we get $\delta(a^{'''}_{p-1})=0$. This means:
\begin{equation}\label{obs2}
a^{'''}_{p-1}\in F_{2d}^\delta. 
\end{equation}
Suppose now $u_0=y_s$, $s\leq d-1$, then $a'_{p,u_0}=0$. So we can assume $u_0=x_s$, $s\leq d-1$. In this case we get $\delta(a'_{p,x_s})=0$, i.e.:
\begin{equation}\label{obs3}
a'_{p,x_s}\in F_{2d}^\delta.\end{equation}
By (\ref{obs1}), (\ref{obs2}) and (\ref{obs3}), using induction, we get \[a_p=\sum b_1(X',V',W',Z'),\]\[a^{'''}_{p-1}=\sum b_2(X',V',W',Z'),\]\[a'_{p,x_i}=\sum b_{3,i}(X',V',W',Z').\]

If $\phi_\alpha(f)=0$ for all $\alpha=(\alpha_1,\ldots,\alpha_{d-1})$, then by Corollary \ref{usefulcorollary} we get $d=2$ and $f$ is in the left ideal generated by $v_{12}$ and $[v_{12},u]$, $u\in U$. Hence $f=fv_{12}+\sum_{u\in U}f_u[v_{12},u]$ and, after a simple manipulation, we apply inductive arguments to $f$ and $f_u$'s and we are done. Now we consider the case $\phi_\alpha(f)\neq0$ for some non-zero $\alpha\in K^{n-1}$. The next two remarks are crucial. First we have \[\deg_X(f)=\deg_{X'}(\phi_\alpha(f))\geq \deg_{Y'}(\phi_\alpha(f))=\deg_{Y}(f).\] Hence \[\deg_X(f)=\deg_{X'}(a_p)\geq\deg_Y(f)=p+\deg_{Y'}(a_p)\] so \[a_p=\sum b'_1(X',V',W',Z')x_{k_1}\cdots x_{k_p}.\] We can argue analogously and obtain \[a^{'''}_{p-1}=\sum b'_2(X',V',W',Z')x_{l_1}\cdots x_{l_p},\]\[a'_{p,x_i}=\sum b_{3,i}(X',V',W',Z')x_{m_1}\cdots x_{m_p}.\]

Now we apply Lemma \ref{finalyes} and Remark \ref{remark} in order to rewrite $a_py_d^p$, $a^{'''}_{p-1}[x_d,y_d]y_d^{p-2}$ and $a'_{p,x_i}[x_i,y_d]y_d^{p-1}$ and we get \[f=c+gx_d+\sum_{i=1}^{d-1}m_{p,i}y_d[x_i,x_d],\] where $c\in\mathcal{C}$. We can rewrite the last summand as $g_1x_d+\sum_{i=1}^{d-1}h_i[x_i,x_d]$, where $\deg_{x_d}h_i=0$ and we get finally \[f=c+g'x_d+\sum_{i=1}^{d-1}h_i[x_i,x_d].\] Hence $g'$ and the $h_i$'s are constants too because $x_d$ and $[x_i,x_d]$ are. Now we are allowed to apply induction on $g'$ and the $h_i$'s because their degrees with respect to the set $\{x_d,y_d\}$ are strictly smaller than the one of $f$ and we conclude the proof.
\end{proof}


\begin{thebibliography}{99}

\bibitem{B}
Yu. A. Bahturin,
Identical Relations in Lie Algebras (Russian),
Nauka, Moscow, 1985.
Translation: VNU Science Press, Utrecht, 1987.

\bibitem{Bed}
L. Bedratyuk,
A note about the Nowicki conjecture on Weitzenbock derivations,
Serdica Math. J. {\bf 35} (2009), 311-316.

\bibitem{DDF}
R. Dangovski, V. Drensky, \c S. F{\i}nd{\i}k,
Weitzenb\"ock derivations of free metabelian Lie algebras,
Linear Algebra Appl.,  {\bf 439} (10) (2013), 3279-3296.

\bibitem{DDF1}
R. Dangovski, V. Drensky, \c S. F{\i}nd{\i}k,
Weitzenb\"ock derivations of free associative algebras,
Journal Algebra Appl.,  {\bf 16} (03) (2017), 1750041/1-26.

\bibitem{DK}
H. Derksen, G. Kemper,
Computational Invariant Theory,
Encyclopaedia of Mathematical Sciences,
Invariant Theory and Algebraic Transformation Groups {\bf 130},
Springer-Verlag, Berlin, 2002.

\bibitem{D}
V. Drensky,
Free Algebras and PI-Algebras,
Springer, Singapore, 1999.

\bibitem{D1}
V. Drensky,
Invariants of unipotent transformations
acting on noetherian relatively free algebras,
{\it Serdica Math. J.} {\bf 30} (2004), 395--404.

\bibitem{DF1}
V. Drensky, \c S. F{\i}nd{\i}k,
The Nowicki conjecture for free metabelian Lie algebras, Int. J. Algebra Comp. (to appear).

\bibitem{DG}
V. Drensky, C.K. Gupta,
Constants of Weitzenb\"ock derivations and invariants of unipotent transformations acting on relatively free algebras,
J. Algebra {\bf 292} (2) (2005), 393-428.

\bibitem{DML}
V. Drensky, L. Makar-Limanov,
The conjecture of Nowicki on Weitzenb\"ock derivations of polynomial algebras,
J. Algebra Appl. {\bf 8} (2009), 41-51.

\bibitem{giz3} A. Giambruno, M. V. Zaicev, A characterization of varieties of associative algebras of exponent two, Serdica Math. J. {\bf 26} (2000), 245-252.

\bibitem{giz1} A. Giambruno, M. V. Zaicev, Exponential codimension growth of P.I. algebras: an exact estimate, Adv. Math. {\bf 142} (1999), 221-243.

\bibitem{giz2} A. Giambruno, M. V. Zaicev, On codimension growth of finitely generated associative algbras, Adv. Math. {\bf 140} (1998), 145-155.

\bibitem{kem2} A. R. Kemer, Ideals of Identities of Associative Algebras, Translations of Math. Monographs 84, AMS, Providence, RI, 1991.

\bibitem{kem1} A. R. Kemer, Varieties and $\Z_2$-graded algebras, Math. USSR Izv. {\bf 25} (1985), 359-374.

\bibitem{K1}
J. Khoury,
Locally Nilpotent Derivations and Their Rings of Constants,
Ph.D. Thesis, Univ. Ottawa, 2004.

\bibitem{K2}
J. Khoury,
A Groebner basis approach to solve a conjecture of Nowicki,
J. Symbolic Comput. {\bf 43} (2008), 908-922.

\bibitem{KR} D. Krakovski, A. Regev, The polynomial identities of the Grassmann algebra, Trans. Amer. Math. Soc. {\bf 181} (1973), 429-438.

\bibitem{Kurano}
K. Kurano,
Positive characteristic finite generation of symbolic Rees algebras and Roberts' counterexamples to the fourteenth problem of Hilbert,
Tokyo J. Math. {\bf 16} (2) (1993), 473-496.

\bibitem{Kuroda}
S. Kuroda,
A simple proof of Nowicki's conjecture on the kernel of an elementary derivation,
Tokyo J. Math. {\bf 32} (2009), 247-251.

\bibitem{L}
J. Lewin,
A matrix representation for associative algebras I,
{\it Trans. Amer. Math. Soc.} {\bf 188} (1974), 29-308.

\bibitem{N}
A. Nowicki,  Polynomial Derivations and Their Rings of Constants,
Uniwersytet Mikolaja Kopernika, Torun, 1994.
www-users.mat.umk.pl/\~{}anow/ps-dvi/pol-der.pdf.

\bibitem{R}
P. Roberts,
An infinitely generated symbolic blow-up in a power series ring and a new counterexample to Hilbert's fourteenth problem,
J. Algebra {\bf 132} (2) (1990), 461-473.

\bibitem{Sh}
A.L. Shmel'kin,
Wreath products of Lie algebras and
their application in the theory of groups (Russian),
Trudy Moskov. Mat. Obshch. {\bf 29} (1973), 247-260.
Translation: Trans. Moscow Math. Soc. {\bf 29} (1973), 239-252.

\bibitem{St}
B. Sturmfels,
Algorithms in Invariant Theory. 2nd ed.
Texts and Monographs in Symbolic Computation, Springer-Verlag, Wien, 2008.

\bibitem{W}
R. Weitzenb\"ock,
\"Uber die Invarianten von linearen Gruppen,
Acta Math. {\bf 58} (1932), 231-293.

\bibitem{Z1} M. V. Zaitsev, Integrality of exponents of growth of identities of finite-dimensional Lie algebras, Izv. Math. {\bf 66} (2002), 463-487.




\end{thebibliography}
\end{document}